\documentstyle[11pt, psfig]{article}
\newtheorem{Theorem}{Theorem}
\newtheorem{Lemma}{Lemma}

\newenvironment {proof} {{\bf Proof.}}{\hspace*{\fill}$\Box$\par\vspace{4mm}}

\def\n{\newblock}
\def\b{\bibitem}

\begin{document}

\title{\bf On the Spectrum of Middle-Cubes\thanks{Research
supported by National Natural Science Foundation of China
(No.~50676091) and Program for New Century Excellent Talents in
University (NCET-06-0546).}}

\author{Ke Qiu${}^a$, \ Rong Qiu${}^b$, \ Yong Jiang${}^b$,
\ Jian Shen${}^c$\\
\small ${}^a$Department of Computer Science, Brock University,
St.~Catharines, Ontario, Canada\\
\small ${}^b$State Key Laboratory of Fire Science, University of
Science and Technology of China\\
\small Hefei, Anhui 230026, P.R.~China\\
\small ${}^c$Department of Mathematics, Texas State University, San
Marcos, TX 78666, USA }

\date{\empty}

\maketitle

\begin{abstract} A middle-cube is an induced subgraph consisting of
nodes at the middle two layers of a hypercube. The middle-cubes are
related to the well-known Revolving Door (Middle Levels) conjecture.
We study the middle-cube graph by completely characterizing its
spectrum. Specifically, we first present a simple proof of its
spectrum utilizing the fact that the graph is related to Johnson
graphs which are distance-regular graphs and whose eigenvalues can
be computed using the association schemes. We then give a second
proof from a pure graph theory point of view without using its
distance regular property and the technique of association schemes.
\end{abstract}

\section{Introduction}

The $n$-dimensional {\em hypercube}, $Q_n$, has $2^n$ nodes such
that two nodes $u$ and $v$, $0 \leq u, v \leq 2^n - 1$, are
connected if and only if their binary representations differ in
exactly one bit. For an odd $n = 2k + 1$, the {\em middle-cube},
$M_n$, is the subgraph induced by all the nodes whose binary
representations have either $k$ 1's or $k+1$ 1's. Fig.~1 shows a
3-cube with its middle-cube highlighted.

\begin{figure}
\centering
\includegraphics[width=1.52in]{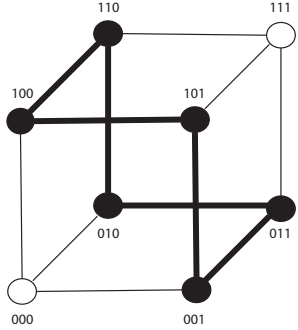}
\caption{A 3-cube}
\end{figure}

The middle-cube $M_{2k+1}$ consists of the nodes at the middle two
layers of the corresponding hypercube $Q_{2k+1}$. Equivalently,
these nodes are at the middle levels $k$ and $k+1$ of the Boolean
lattice ${\cal B}_{2k+1}$ (and the Hasse diagram of ${\cal
B}_{2k+1}$ is isomorphic to $Q_{2k+1}$) \cite{Shields}. Middle-cubes
have been considered as a possible topology to interconnect
processors in networks \cite{Madabhushi}. A well-known open problem
concerning middle-cubes is {\em the Revolving Door (Middle Levels)
conjecture} \cite{Havel, West_Problem}: All middle-cubes $M_{2k+1}$
are Hamiltonian. The conjecture has been verified for $k \le 17$
\cite{Shields} but remains open in general. Partial results on the
conjecture can be found in \cite{Johnson, Savage, Shields}. In
particular, Johnson proved in 2004 that $M_{2k+1}$ has a cycle of
length $(1-o(1))|M_{2k+1}|$, where $|M_{2k+1}|= 2 { 2k +1 \choose
k}$ is the number of vertices in $M_{2k+1}$.

The {\em spectrum} of a graph consists of all distinct eigenvalues
and their respective multiplicities of the adjacency matrix of the
graph. It is worth mentioning that the spectral and structural
properties of a graph are related \cite{Cvetkovic, West}. For
example, van den Heuvel~\cite{Heuvel} proved some necessary spectral
conditions for a graph to be Hamiltonian. To better understand
various properties for the middle-cubes, it may be necessary to
study the spectrum for middle-cubes. In the next section, we give a
complete characterization for the spectrum of the middle-cubes by
giving two different proofs, one from the distance-regular graph
point of view, and the other from a pure graph theory point of view.

\section{Spectrum of Middle-Cubes}

We always assume $n=2k+1$ throughout the paper. Without confusion
from the context, we abuse the notation $M_n$ for both the middle
cube and its adjacency matrix. The eigenvalues and their
corresponding multiplicities for $M_n$, $n$ = 3, 5, 7, 9, are given
in Table 1. In this section, we will prove that
\begin{Theorem} \label{main}
The characteristic
polynomial of $M_n$ is
$$|\lambda {\bf I} - M_n| = \Pi_{i=1}^{k+1} (\lambda \pm i)^{{n
\choose k+1-i} - {n \choose k-i}}.$$
\end{Theorem}
We also note that the sequence 1, 2, 1, 4, 5, 1, 6, 14, 14, ...
appears in the {\em The On-Line Encyclopedia of Integer Sequences}
as sequence {\em A050166} \cite{Sloane, Guy}.

\begin{table}[th]
\caption{Eigenvalues of Middle-Cubes $M_n$ for $n$ = 3, 5, 7, 9.}
\begin{center}
\begin{tabular}{|c|c|c|c|c|c|c|c|c|c|c|}        \hline
  & -5 & -4 & -3 & -2 & -1 & 1 & 2 & 3 & 4 & 5 \\ \hline
3 &    &    &    &  1 &  2 & 2 & 1 &   &   &   \\ \hline
5 &    &    & 1  &  4 & 5  & 5 & 4 & 1 &   &   \\ \hline
7 &    &  1 & 6  & 14 & 14 &14 &14 & 6 & 1 &   \\ \hline
9 & 1  &  8 & 27 & 48 &42  &42 &48 &27 & 8 & 1 \\ \hline
\end{tabular}
\end{center}
\label{middlecube}
\end{table}

We will first consider the middle cube
by relating it to Johnson graphs.
Let $X$ be a finite set and $e$ an positive integer. The {\it
Johnson graph of the e-sets in $X$} has a vertex set $X \choose e$,
the set of all $e$-subsets of $X$ (subsets of cardinality $e$). Two
vertices $u$ and $v$ are adjacent whenever $|u \cap v|$ = $e-1$
\cite{Brouwer1}. Since Johnson graphs are distance regular, the
eigenvalues of Johnson graphs can be computed using association
schemes described in \cite[Chapter 2]{Brouwer1} and \cite[Pages
69-72]{Brouwer2}.

\begin{proof}
Let $M_{2k+1}$ be the adjacency matrix of the middle-cube. Let
$J(n,m)$ be the adjacency matrix of the Johnson graph with vertex
set ${ [n] \choose m }$, where two $m$-subsets are adjacent when
they have exactly $m-1$ elements in common. Then by the definition
of $M_{2k+1}$ and $J(n,m)$,
$$M_{2k+1}^2 = \left [ \begin{array}{cc}
J(2k+1, k+1) & O \\
O & J(2k+1, k+1)
\end{array} \right ] + (k+1) I.$$
By \cite[Page 79]{Brouwer2}, $J(n, m)$ has eigenvalues
$(m-i)(n-m-i)-i$ with multiplicity ${n \choose i} - { n \choose
{i-1}}$. (The proof of the spectrum of $J(n, m)$ requires techniques
using the association schemes.) Then $M_{2k+1}^2$ has eigenvalues
$$(k+1-i)(k-i)-i+ (k+1) = (k+1-i)^2$$
with multiplicity $2\left ( {n \choose i} - { n \choose {i-1}}
\right ).$ Since $M_{2k+1}$ is bipartite, it has symmetric
positive/negative eigenvalues. Therefore $M_{2k+1}$ has eigenvalues
$\pm (k+1-i)$ with multiplicity ${n \choose i} - { n \choose
{i-1}}$.
\end{proof}

We now study the spectrum of the middle cube from a graph
theoretical point of view. Before we give a second proof for
Theorem~\ref{main} using graph theory, several definitions are in
order. We denote by $S$ the set $\{1,2,\ldots, n\}$. We use $A+B$ to
denote the union of the sets $A$ and $B$. Similarly, $A-B$ denotes
the set of elements that are in $A$ but not in $B$, while $|A|$
represents the size of the set $A$. An $r$-set is a set of size $r$.
W denote by ${ S \choose k }$ the set of all $k$-subsets of $S$. For
the convenience of our proofs, we also view ${ S \choose k } + { S
\choose k +1 }$ as the vertex set of the middle-cube $M_n$. Thus the
edge set of $M_n$ is induced by the inclusion relation; that is, two
distinct vertices $A, B$ in ${ S \choose k } + { S \choose k  +1 }$
are adjacent if and only if $A \subset B$ or $B \subset A$.

For each positive integer $i$, let $A_1^{(i)}, A_2^{(i)}, \ldots,
A_{ { n \choose i }}^{(i)} $ be an ordering of all ${ n \choose i }$
$i$-subsets of $S$. Let $r$ be a fixed positive integer with $r \le
k$. Let $x_1, \ldots, x_{{n \choose r}}$ be real variables. Define a
weight function:
$$f(A_i^{(r)}) = x_i, \  1 \le i \le { n \choose r }$$
subject to the following ${ n \choose r-1 }$ constraints
\begin{equation} \label{constraints}
\sum _{ i \not \in R} f(R + \{ i \}) =0 \mbox { for each } R \in {S
\choose r-1}. \end{equation} For each $A \subseteq S$, we define
\begin{equation} \label{function}
f(A) = \sum_ {A_i ^{(r)}\in { A \choose r } } f\left
(A_i^{(r)}\right ).
\end{equation}
(Thus $f(A) =0$ whenever $|A| \le r-1$.) For each $i$ with $r \le i
\le k$, we define
$$V_i^{(r)} = \left \{\left ( f \left (A_1^{(i)} \right ), \ldots, f \left (A_{ { n
\choose i }}^{(i)} \right ) \right ) \right \}$$ and $$ V_{i,
i+1}^{(r)} = \left \{ \left ( f \left (A_1^{(i)} \right ), \ldots, f
\left (A_{ { n \choose i }}^{(i)} \right ) , f \left (A_1^{(i+1)}
\right ), \ldots, f \left (A_{ { n \choose i+1 }}^{(i+1)} \right )
\right ) \right \} .$$ Then both $(V_i^{(r)}, +)$ and
$(V_{i,i+1}^{(r)}, +)$ are vector spaces on reals.
We will show that $V_{k,k+1}^{(r)}$ is the eigenspace with dimension
${ n \choose r} - { n \choose r-1 }$ corresponding to the eigenvalue
$k+1-r$ for the matrix $M_n$.

\begin{Lemma}\label{Lemma1} Let $A$ be a subset of $S$.
Suppose $i \not \in A$ and $j \in A$. Then
$$f(A+\{i\}) = f(A) + \sum _{R \in { A \choose r-1} } f(R+ \{ i\})$$
and
$$f(A-\{j\}) = f(A) - \sum _{R \in { A-\{j\} \choose r-1} } f(R+ \{
j\}).$$
\end{Lemma}

\begin{proof}
$$ \begin{array}{ll}
f(A + \{ i \}) & = \sum _ { R \in { A + \{ i \} \choose r }} f(R) \\
& = \sum_ {i \not \in R \in {A + \{ i \} \choose r }} f(R) + \sum_
{i \in R \in { A + \{ i \} \choose r }} f(R)\\
& = \sum_ { R \in { A \choose r }} f(R) + \sum_
{R-\{i\} \in { A \choose r -1 }} f(R)\\
&= f(A) +\sum_ {R\in { A \choose r -1 }} f(R+\{i\}).
\end{array}$$
Similarly,
$$\begin{array}{ll}
f(A-\{j\}) & = \sum _ { R \in { A - \{ j \} \choose r }} f(R) \\
&=\sum _ { R \in { A \choose r }} f(R) - \sum _{ j \in R \in { A
\choose r}} f(R) \\
&=f(A) -\sum _{ R-\{j\} \in { A-\{j\}
\choose r-1}} f(R) \\
&= f(A)-\sum _{R \in { A-\{j\} \choose r-1} } f(R+ \{ j\}).
\end{array}$$
\end{proof}

\begin{Lemma} \label{Lemma2}
Let $A$ be a subset of $S$. Then
$$\sum_{ i \not \in A} \sum _{ R \in { A \choose r-1}} f(R+\{i\}) =
-r f(A).$$
\end{Lemma}

\begin{proof} The lemma is trivial if $|A| \le r-1$ (in which case $f(A)=0$).
Suppose now $|A| \ge r$. By (\ref{constraints}) and
Lemma~\ref{Lemma1},
$$\begin{array}{ll}
0 & = \sum _{ R \in { A \choose r-1}} \sum _{ i \not \in R}
f(R+\{i\}) \\
& =\sum _{ R \in { A \choose r-1}} \left ( \sum _{i \in A-R }
f(R+\{i\}) +\sum _{i \not \in A } f(R+\{i\})\right ) \\
& =\sum _{ R \in { A \choose r-1}} \sum _{R \subset B \in { A
\choose r }} f(B) + \sum _{ R \in { A \choose r-1}} \sum _{ i \not
\in A}f(R+\{i\})\\
& =\sum _{B \in { A \choose r }} \sum _{ R \in { B \choose r-1}}
f(B) +\sum _{ i \not \in A} \sum _{ R \in { A \choose r-1}}
f(R+\{i\})\\
&= r\sum _{B \in { A \choose r }}f(B) +\sum _{ i \not \in A} \sum _{
R \in { A \choose r-1}}
f(R+\{i\})\\
&= rf(A) + \sum _{ i \not \in A} \sum _{ R \in { A \choose r-1}}
f(R+\{i\}),
\end{array}$$
from which Lemma~\ref{Lemma2} follows.
\end{proof}

\begin{Lemma} \label{Lemma3}
Let $A$ be a subset of $S$. Then
$$\sum_{ i \in A} \sum _{ R \in { A -\{i\} \choose r-1}} f(R+\{i\}) =
r f(A).$$
\end{Lemma}

\begin{proof} The lemma is trivial if $|A| \le r-1$ (in which case $f(A)=0$).
Suppose now $|A| \ge r$. By (\ref{constraints}), Lemmas~\ref{Lemma1}
and \ref{Lemma2},
$$\begin{array}{ll}
0 & =\sum _{ R \in { A \choose r-1}} \sum _{ i \not \in R}
f(R+\{i\}) \\
& =\sum _{ R \in { A \choose r-1}} \left ( \sum _{i \not \in A}
f(R+\{i\}) +\sum _{i \in A-R } f(R+\{i\}) \right ) \\
& =\sum _{ R \in { A \choose r-1}}\sum _{i \not \in A} f(R+\{i\}) +
\sum _{ R \in { A \choose r-1}}  \sum _{i \in A-R } f(R+\{i\})\\
&=\sum _{i \not \in A} \sum _{ R \in { A \choose r-1}}f(R+\{i\}) +
\sum _{i \in A} \sum _{ R \in { A-\{i\} \choose r-1}}  f(R+\{i\})\\
&= -rf(B) +\sum _{i \in A} \sum _{ R \in { A-\{i\} \choose r-1}}
f(R+\{i\}),\end{array}$$ from which Lemma~\ref{Lemma3} follows.
\end{proof}

Recall that $V_{k,k+1}^{(r)} =\left \{ \left ( f \left (A_1^{(k)}
\right ), \ldots, f \left (A_{ { n \choose k }}^{(k)} \right ) ,
f\left (A_1^{(k+1)} \right ), \ldots, f \left (A_{ { n \choose k+1
}}^{(k+1)} \right ) \right ) \right \} $ is a vector space on reals.
Let $E_{\lambda} $ be the eigenspace corresponding to the eigenvalue
$\lambda$ for the matrix $M_n$.

\begin{Lemma} \label{Lemma4}
Let $n=2k+1$ and $1 \le r \le k$. Then $V_{k,k+1}^{(r)} \subseteq
E_{k+1-r}.$
\end{Lemma}

\begin{proof} For any vertex $A \in { S \choose k} + {S \choose k+1}$ in the middle-cube $M_n$, let
$ \Gamma (A)$  be the neighbor set of $A$. Then
$$ \Gamma (A) = \left \{ \begin{array}{ll}
\{ A +\{i\}: i \not \in A\} & \mbox { if } A \in {S \choose k}; \\
\{ A - \{i\}: i \in A\} & \mbox { if } A \in {S \choose k+1}.
\end{array} \right . $$
Thus to prove the lemma, it suffices to prove the following two
identities:
\begin{equation} \label{equation2}
\sum_{B \in \Gamma (A)} f(B)= \sum_{i \not \in A} f(A+\{i\}) =
(k+1-r) f(A) \mbox { for each } A \in {S \choose k}
\end{equation}
\begin{equation} \label{equation3}
\sum_{B \in \Gamma (A)} f(B)=\sum_{i \in A} f(A- \{i\}) = (k+1-r)
f(A) \mbox { for each } A \in {S \choose k+1}
\end{equation}
Proof of (\ref{equation2}): By Lemmas \ref{Lemma1} and \ref{Lemma2},
$$\begin{array}{ll}
\sum_{i \not \in A} f(A+\{i\}) &= \sum_{i \not \in A} \left ( f(A) +
\sum _{R \in { A \choose r-1} } f(R+ \{ i\}) \right ) \\
&=\sum_{i \not \in A} f(A) +\sum_{i \not \in A}\sum _{R \in { A
\choose r-1} } f(R+ \{ i\})\\
&=(n-|A|) f(A) - rf(A) \\
&=(n-|A|-r) f(A) =(k+1-r) f(A).
\end{array}$$
Proof of (\ref{equation3}): By Lemmas \ref{Lemma1} and \ref{Lemma3},
$$\begin{array}{ll}
\sum_{i \in A} f(A- \{i\}) & = \sum_ {i \in A} \left (  f(A) -\sum
_{R \in { A-\{i\} \choose r-1} } f(R+ \{ i\}) \right ) \\
&=\sum _{i \in A}  f(A)- \sum _{i \in A}\sum _{R \in { A-\{i\}
\choose
r-1} } f(R+ \{ i\}) \\
&= |A| f(A) -r f(A) \\
&= (k+1-r) f(A).
\end{array}$$
\end{proof}

\begin{Lemma} \label{Lemma5} Let $n =2k+1$ and $1 \le r \le k$. Then
$${\rm{dim }} \ V_{k, k+1}^{(r)} = {n \choose r} -{n \choose r-1}.$$
\end{Lemma}

\begin{proof} Let $M_{i,j}$ be the incidence matrix whose rows
correspond to the $i$-subsets $A_1^{(i)}, A_2^{(i)}, \ldots, A_{ { n
\choose i }}^{(i)} $, and whose columns correspond to the
$j$-subsets $A_1^{(j)}, A_2^{(j)}, \ldots, A_{ { n \choose j}}^{(j)}
$; that is, the $(r,s)$-entry of $M_{i,j}$ is $1$ if $A_r^{(i)}
\subset A_s^{(j)}$ or $A_s^{(j)} \subset A_r^{(i)},$ and $0$
otherwise. By \cite[Corollary 2]{Gottlieb}, the matrix $M_{i,j}$ has
full rank; that is,
$$ \mbox{rank } M_{i,j} = \min \left \{ {n \choose i}, {n \choose j} \right \}.$$
Recall the definition that $V_r^{(r)} = \left \{\left ( f \left
(A_1^{(r)} \right ), \ldots, f \left (A_{ { n \choose r }}^{(r)}
\right ) \right ) \right \}.$ By (\ref{constraints}), $V_r^{(r)}$
consists of all solution sets to the following homogeneous matrix
equation:
$$\left (x_1,x_2, \ldots, x_{{n \choose r }} \right ) M_{r, r-1} = (
0, 0, \ldots, 0).$$ Thus $$\mbox{dim } V_r ^{(r)} = {n \choose r} -
\mbox{rank } M_{r,r-1} = {n \choose r} - {n \choose r-1}.$$ By
(\ref{function}) and the definition of $V_{k,k+1}^{(r)}$, each
vector in $V_{k,k+1}^{(r)}$ can be written as
$$\left (x_1,x_2, \ldots,
x_{{n \choose r }} \right ) \left [ M_{r, k} \ \ \vdots \ \ M_{r,
k+1}\right ]$$ for some vector $\left (x_1,x_2, \ldots, x_{{n
\choose r }} \right ) \in V_r^{(r)}$. This implies that $V_{k,
k+1}^{(r)} = V_r ^{(r)} \left [ M_{r, k} \ \ \vdots \ \ M_{r,
k+1}\right ].$ Thus
$$\begin{array} {ll}
\mbox{dim } V_r ^{(r)} & \ge \mbox{dim } V_{k,k+1}^{(r)} \ge
\mbox{dim } V_r ^{(r)}+ \mbox{rank } \left [ M_{r, k} \ \ \vdots \ \
M_{r, k+1}\right ] - {n \choose r} \vspace{.2cm} \\
& \ge \mbox{dim } V_r^{(r)} + \mbox{rank } M_{r, k} - {n \choose
r}\vspace{.1cm}\\
&=\mbox{dim } V_r^{(r)} +  \min \left \{ {n \choose r}, {n \choose
k} \right \} - {n \choose r}\vspace{.1cm}\\
& =\mbox{dim } V_r ^{(r)} \end{array} $$ and so
$$\mbox{dim
} V_{k,k+1} ^{(r)}= \mbox{dim } V_r ^{(r)}= { n \choose r} -{n
\choose r-1}.$$
\end{proof}

\begin{Theorem} Let $n =2k+1$ and $1 \le r \le k$. Then
$$E_{k+1-r} = V_{k,k+1}^{(r)}$$
and
$${\rm{dim }} \ E_{k+1-r} = {\rm{dim }} \ E_{r-k-1}= {n \choose r} -{n
\choose r-1}.$$ Furthermore, the characteristic polynomial of the
matrix $M_n$ is $$|\lambda {\bf I} - M_n| = \Pi_{i=1}^{k+1} (\lambda
\pm i)^{{n \choose k+1-i} - {n \choose k-i}}.$$
\end{Theorem}

\begin{proof} The equation $\mbox{dim } E_{k+1-r} = \mbox{dim } E_{r-k-1}$ holds since the
middle-cube $M_n$ is a bipartite graph. Since $M_n$ is a connected
$(k+1)$-regular graph, we have $\mbox{dim } E_{k+1} =1$.
By Lemmas \ref{Lemma4} and \ref{Lemma5},
$$\begin{array}{ll}
{n \choose k} + {n \choose k+1} & \ge \sum_{r=0}^k \left ( \mbox{dim
} E_{k+1-r} +\mbox{dim } E_{r-k-1} \right ) \vspace{.1cm} \\
& = 2\sum_{r=0}^k \mbox{dim } E_{k+1-r} \vspace{.1cm}\\
& \ge 2+ 2\sum_{r=1}^k \mbox{dim } V_{k,k+1} ^{(r)}\vspace{.1cm} \\
& =  2+2\sum_{r=1}^k \left ( { n \choose r} - { n \choose r-1 }
\right ) \vspace{.1cm} \\
& = 2 {n \choose k} ={n \choose k} + {n \choose k+1}. \end{array}$$
Thus all equalities hold throughout. This also implies that all
eigenvalues of $M_n$ are integers $i$ with $1 \le |i | \le k+1$, and
that each eigenvalue of $i$ has multiplicity ${n \choose k+1-i}-{ n
\choose k-i}$, where ${n \choose -1} =0$.
\end{proof}

\section{Conclusion}

We prove that the characteristic polynomial of the middle-cube $M_n$
with $n=2k+1$ is
$$\Pi_{i=1}^{k+1} (\lambda \pm i)^{{n
\choose k+1-i} - {n \choose k-i}}.$$ This spectral property may be
useful in future research on various properties of the middle-cubes.

\end {document}